\documentclass[12pt]{article}
\usepackage{latexsym, amsmath, amsfonts, amscd, amssymb, verbatim, amsxtra,amsthm}
\usepackage{pgf,tikz,pgfplots}
\usepackage{tkz-euclide}
\usepackage{graphicx}
\usepackage{color}
\usepackage{hyperref}
\usepackage[all]{xy}
\usepackage{mathrsfs}
\usepackage{tikz,tkz-tab}
\usepackage{parallel}
\usepackage{array}
\usepackage{xcolor}
\usepackage{verbatim}
\usepackage{tikz-3dplot}
\usepackage{comment}
\usepackage{tikz-cd}
\usepackage{enumerate}
\usepackage{subcaption}

\usetikzlibrary{matrix}
\usetikzlibrary{calc,intersections}
\usetikzlibrary{patterns}        
\pgfplotsset{compat=1.15}
\usetikzlibrary{arrows}
\DeclareGraphicsExtensions{.pdf,.png,.jpg}

\begin{document}
	\pagestyle{myheadings}
	\newtheorem{theorem}{Theorem}[section]
	\newtheorem{lemma}[theorem]{Lemma}
	\newtheorem{corollary}[theorem]{Corollary}
	\newtheorem{proposition}[theorem]{Proposition}
	
	\theoremstyle{definition}
	\newtheorem{definition}[theorem]{Definition}
	\newtheorem{example}[theorem]{Example}
	\newtheorem{xca}[theorem]{Exercise}
	\newtheorem{algo}{Algorithm}
	\theoremstyle{remark}
	\newtheorem{remark}[theorem]{Remark}
	
	\numberwithin{equation}{section}
	\renewcommand{\theequation}{\thesection.\arabic{equation}}
	\normalsize
	
	\setcounter{equation}{0}
	
	\title{  \bf  On Separation of level sets for a pair of quadratic functions}
	
		\author{Huu-Quang Nguyen\footnote{Department of Mathematics, Vinh
			University; Nghe An, Vietnam, email:
			quangdhv@gmail.com.},\ \,   Ya-Chi Chu, and ~Ruey-Lin Sheu\footnote{Department of Mathematics, National Cheng Kung University, Tainan, Taiwan; email: rsheu@mail.ncku.edu.tw}}
	\maketitle	
	\makeatletter   \renewcommand\@biblabel[1]{#1.}
	\makeatother
	\medskip
	
\begin{quote}
\small{\bf Abstract.}\; Given a quadratic function $f(x)=x^TAx+2a^Tx+a_0,$ it is possible that its level set $\{x\in\mathbb{R}^n: f(x)=0\}$ has two connected components and thus can be separated by the level set $\{x\in\mathbb{R}^n: g(x)=0\}$ of another quadratic function $g(x)=x^TBx+2b^Tx+b_0.$ It turns out that the separation property of such kind has great implication in quadratic optimization problems and thus deserves careful studies. 
In this paper, we characterize the separation property analytically by necessary and sufficient conditions as a new tool to solving optimization problems.
%
%
%
%
\medskip
		
\vspace*{0,05in} {\bf Key words.}\; Quadratic Mapping, Level Sets, Separation, Connectedness, Convexity, Joint Numerical Range, Quadratic Optimization.
		
{\bf Mathematics Subject Classification (2010).} 90C20,	90C22, 90C26.
\end{quote}

\section{Introduction} \label{sec:Intro}

Let both $f(x)=x^TAx+2a^Tx+a_0$ and $g(x)=x^TBx+2b^Tx+b_0$ be non-constant functions, where $A,~B$ are $n\times n$ symmetric matrices, $a,b\in\mathbb{R}^n$ and $a_0,b_0\in\mathbb{R}$. In \cite[2019]{Quang-Sheu18}, Nguyen and Sheu defined the concept of separation between two sets as:
\begin{definition}[\cite{Quang-Sheu18}] \label{def-separation--}
The level set $\{x\in\mathbb{R}^n |~g(x)=0\}(\triangleq\{g=0\})$ is said to separate the set $\{x\in\mathbb{R}^n|~f(x) \star 0\}(\triangleq\{f\star0\})$, where $\star \in \{<,\le,=\}$ if there are non-empty subsets $L^-$ and $L^+$ of $\{f\star0\}$ such that
\begin{eqnarray}
&&\{f \star 0\} = L^-\cup L^+; \label{eq:def_sep-1}\\
&& g(a^-)g(a^+)<0,~\forall~ a^-\in L^-;~\forall~a^+\in L^+. \label{eq:def_sep-2}
\end{eqnarray}
\end{definition}
\noindent They proved the particular case that $\{g=0\}$ separates $\{f<0\}$ (namely, substitute $\star$ with $<$ in Definition \ref{def-separation--}) when, and only when $\{f<0\}$ has two connected components;
$g(x)$ reduces to an affine function; and the level set
$\{g=0\}\subset \{f\ge0\}.$ Immediately, it can be noted that the inclusion relation $\{g=0\}\subset \{f\ge0\}$ is closely related to the $\mathcal{S}$-lemma with equality \cite[2016]{Xia-Wang-Sheu16}, which studies under what conditions the following two statements be equivalent $\big(({\rm E_1})\sim ({\rm E_2})\big):$
\begin{itemize}
    \item[](${\rm E_1}$)~~ ($\forall x\in\mathbb R^n$)
    $~g(x)= 0~\Longrightarrow~ f(x)\ge 0.$ 
    \item[](${\rm E_2}$)~~   ($\exists \lambda\in\mathbb R$)
    ($\forall x\in\mathbb R^n$)~$f(x) + \lambda g(x)\ge0.$ 
\end{itemize}
Surprisingly, it was shown in Theorem 2 \cite[2019]{Quang-Sheu18} that $\{g=0\}$ separates $\{f<0\}$ if and only if the two statements fail to be equivalent $[\text{i.e.
}({\rm E_1})\not\sim ({\rm E_2})]$. Since the $\mathcal{S}$-lemma with equality by Xia et. al \cite[2016]{Xia-Wang-Sheu16} has been successfully used to solve the long-standing non-convex quadratic optimization problem $\min\{f(x): g(x)=0\},$ and the proof through the idea of separation of level sets as in \cite[2019]{Quang-Sheu18} is much shorter than that in the original proof by Xia et. al \cite[2016]{Xia-Wang-Sheu16}, the evidence makes us believe that the separation property of two quadratic functions is fundamental to the strong duality for non-convex quadratic optimization problems. It is therefore natural to ask the other two cases: when $\{g=0\}$ can separate $\{f\le0\};$ and when $\{g=0\}$ can separate $\{f=0\}?$

Separation of $\{g=0\}$ to $\{f=0\}$ is quite different from $\{g=0\}$ to separate $\{f<0\}$. The difference can be easily noticed by several examples.
In Figure 1, $f(x,y) = -x^2 + 4 y^2$ and $g(x,y) = 2x-y$. Then, $\{f<0\}$ consists of two open disjoint branches which share a common boundary point $(0,0).$ Then, $\{g=0\}$ passes the common boundary point and separate the two branches of $\{f<0\}$. However, $\{f=0\}$ is the intersection of two lines, which is connected and thus cannot be separated by $\{g=0\}$. In Figure 2, $f(x,y) = -x^2 + 4 y^2 - 1$ and $g(x,y) = x-5y$. Since $\{g=0\}\subset \{f<0\}$ so that
$\{g=0\}$ cannot separate $\{f<0\},$ but $\{g=0\}$ does separate the boundary
of $\{f<0\},$ which is $\{f=0\}.$ In other words, separation of $\{g=0\}$ to $\{f<0\}$ and separation of $\{g=0\}$ to $\{f=0\}$ cannot imply from one to the other.

Convexity of quadratic maps problem and S-procedure problem have attracted much attention from researchers due to their applications as well as interesting in itself. The main results in this field were obtained by Toeplitz \cite[1918]{Toeplitz18}, \cite[1919]{Hausdorff19}, Finsler \cite[1937]{Fin1937}, Dines \cite[1941]{Dines41}, Brickman \cite[1961]{Brickman61}, Yakubovich \cite[1971]{Yakubovich71}, Yuan \cite[1990]{Yuan90}, Polyak \cite[1998]{Polyak98}, Polik et al. \cite[2007]{Polik-Terlaky07}, Beck \cite[2007]{Beck07}, Jeyakumar et al. \cite[2009]{Jeyakumar-Lee-Li09}, Xia et al. \cite[2016]{Xia-Wang-Sheu16}, Flores et al. \cite[2016]{FloresBazan-Opazo16}. Various tools were developed to deal with those problems, however, the proof of strong results were very long and complicate (e.g. see the proofs in \cite{Xia-Wang-Sheu16,FloresBazan-Opazo16}). Moreover, those given tools were not enough to solve some difficult problems. Motivated by this fact, we have proposed a new one called "{\em the separation for the level sets}". This new tool  not only made the proofs shorter, simpler but also  helped us to get stronger results (e.g. see recently results in Nguyen et al. \cite[2018]{Quang-Sheu18}, Nguyen et al. \cite[2020]{Q-C-S2020}, Nguyen et al. \cite[2020]{Q-S-X2020,Q-S-S-X2020}). 


We have the following properties directly implied from the definition:

\begin{itemize}
\item If $\{g=0\}$ separates $\{f*0\}$, $g$ satisfies the two-side Slater condition. Moreover, from \eqref{eq:def_sep-2}, $ 0\not\in g(\{f*0\}).$ Therefore,
      \begin{equation}\label{S-lemma-type}
        \{g=0\} \subset \{f*0\}^c \mbox{ or equivalently, } \{f*0\}\cap \{g=0\}=\emptyset.
      \end{equation}
  \item If $\{g=0\}$ separates $\{f*0\}$, the set $\{f*0\}$ is disconnected. Suppose the contrary that $g(\{f\star0\})$ is connected, which is an interval. Since $0\not\in g(\{f*0\}),$ the interval must be either
      $g(\{f\star0\})<0$ or $g(\{f\star0\})>0.$ Both violate
      \eqref{eq:def_sep-2}, so $\{f \star 0\}$ must be disconnected.
  \item If $\{g=0\}$ separates $\{f*0\}$, the set $\{f*0\}$ has exactly two connected components. The fact that $\{f<0\}$ and $\{f\le0\}$ have two connected components has been proved by Lemma 1 and Lemma 2 in \cite[2019]{Quang-Sheu18}, respectively. The proof for $\{f=0\}$ to have at most two connected components is provided in Lemma \ref{lem:disconnected_f=0} of this article. Let us denote the two components of $\{f*0\}$ by $\tilde{L}^-,~\tilde{L}^+.$
  \item If $\{g=0\}$ separates $\{f*0\}$, since $\{f*0\}$ has two connected components $\tilde{L}^-,\tilde{L}^+$ and $0\not\in g(\{f*0\})$, both $g(\tilde{L}^-)$ and $g(\tilde{L}^+)$ are intervals not containing $0.$ From \eqref{eq:def_sep-2}, $g(\tilde{L}^-)$ and $g(\tilde{L}^+)$ cannot sit on the same side of $0.$ Thus, the two connected components of $\{f*0\}$ are indeed $L^-$ and $L^+$ in Definition \ref{def-separation}.
\end{itemize}

In summary, when $\{g=0\}$ separate $\{f*0\}$ for either $\star \in \{<,\le,=\},$
we know in general that $\{f*0\}$ has two connected components $L^-$ and $L^+$; $g$ satisfies the two-side Slater condition; $g(L^-)g(L^+)<0$ and $\{g=0\}$ lies entirely in the complement of $\{f*0\}$, see
\eqref{S-lemma-type}.

\begin{figure}
\centering
\begin{minipage}[t]{0.485\linewidth}
\centering
\captionsetup{width=\linewidth}
\definecolor{ffqqqq}{rgb}{1.,0.,0.}
\definecolor{wwccff}{rgb}{0.4,0.8,1.}
\begin{tikzpicture}[scale=0.653]
\clip(-5.4,-3.5) rectangle (5.5,3.5);
\draw [-stealth] (-5.5,0)--(5.5,0);
\draw [-stealth] (0,-3.5)--(0,3.5);

\draw[color=blue] (-2.5,0.3) node {\tiny$\{f<0\}$};
\draw[color=blue] (2.5,0.3) node {\tiny$\{f<0\}$};
\draw[color=blue] (-2.5,2.1) node {\tiny$\{f=0\}$};
\draw[color=red] (2.2,2.5) node {\tiny$\{g=0\}$};

\draw[line width=1.pt,dash pattern=on 1pt off 1pt,color=wwccff,fill=wwccff,pattern=north east lines,pattern color=wwccff](-14.4,7.2)--(-14.4,-7.2)--(0.,0.)--(-14.4,7.2);
\draw[line width=1.pt,dash pattern=on 1pt off 1pt,color=wwccff,fill=wwccff,pattern=north east lines,pattern color=wwccff](17.56,8.78)--(0.,0.)--(17.04,-8.52)--(19.68,-8.52)--(19.68,8.78);
\draw [line width=0.7pt,color=ffqqqq,domain=-14.4:19.68] plot(\x,{(-0.-2.*\x)/-1.});
\end{tikzpicture}

\caption{\small Let $f(x,y) = -x^2 + 4 y^2$ and $g(x,y) = 2x-y$. The level set $\{g=0\}$ separates $\{f<0\},$ while $\{g=0\}$ does not separate $\{f=0\}.$}
\label{fig:sep-intro-sublevel-not-level}
\end{minipage}
\hfill
\begin{minipage}[t]{0.485\linewidth}
\centering
\captionsetup{width=\linewidth}

\definecolor{ffqqqq}{rgb}{1.,0.,0.}
\definecolor{wwccff}{rgb}{0.4,0.8,1.}
\begin{tikzpicture}[scale=0.653]
\clip(-5.4,-3.5) rectangle (5.5,3.5);

\draw [-stealth] (-5.5,0)--(5.5,0);
\draw [-stealth] (0,-3.5)--(0,3.5);

\draw[color=blue] (-2.5,0.3) node {\tiny$\{f<0\}$};
\draw[color=blue] (-2.5,2.1) node {\tiny$\{f=0\}$};
\draw[color=red] (3.6,1.13) node {\tiny$\{g=0\}$};

\draw[line width=1.pt,dash pattern=on 1pt off 1pt,color=wwccff,fill=wwccff,pattern=north east lines,pattern color=wwccff](17.530062659352904,8.78)--(11.743793945312508,5.893146057128907)--(8.303625488281256,4.1818112182617195)--(6.363449707031255,3.2207720947265632)--(5.115822753906254,2.6063214111328135)--(4.244561767578129,2.180384521484376)--(3.6004492187500037,1.868370361328126)--(3.1038671875000037,1.6304901123046884)--(2.7084533691406283,1.4435827636718759)--(2.3853747558593783,1.2932531738281259)--(2.115750732421878,1.170085449218751)--(1.8866992187500031,1.067664794921876)--(1.689122314453128,0.9814703369140635)--(1.516413574218753,0.908227539062501)--(1.3636474609375029,0.8455078125000011)--(1.2270788574218778,0.7914739990234385)--(1.1038061523437528,0.7447131347656261)--(0.9915380859375027,0.7041217041015635)--(0.8884387207031277,0.668828125000001)--(0.7930187988281276,0.6381378173828135)--(0.7040466308593776,0.6114904785156261)--(0.6204980468750025,0.5884338378906261)--(0.5415039062500026,0.5686004638671885)--(0.4663183593750025,0.5516912841796886)--(0.39429321289062746,0.5374633789062511)--(0.3248596191406274,0.5257220458984385)--(0.2575024414062524,0.5163110351562511)--(0.19175903320312734,0.509110107421876)--(0.12719604492187733,0.5040283203125011)--(0.0634057617187523,0.5010040283203135)--(0.,0.5)--(-0.06340576171874775,0.5010040283203135)--(-0.12719604492187278,0.5040283203125011)--(-0.19175903320312282,0.509110107421876)--(-0.25750244140624784,0.5163110351562511)--(-0.3248596191406229,0.5257220458984385)--(-0.3942932128906229,0.5374633789062511)--(-0.46631835937499794,0.5516912841796886)--(-0.541503906249998,0.5686004638671885)--(-0.620498046874998,0.5884338378906261)--(-0.704046630859373,0.6114904785156261)--(-0.7930187988281231,0.6381378173828135)--(-0.8884387207031231,0.668828125000001)--(-0.9915380859374981,0.7041217041015635)--(-1.1038061523437481,0.7447131347656261)--(-1.2270788574218732,0.7914739990234385)--(-1.3636474609374982,0.8455078125000011)--(-1.5164135742187483,0.908227539062501)--(-1.6891223144531236,0.9814703369140635)--(-1.8866992187499985,1.067664794921876)--(-2.1157507324218736,1.170085449218751)--(-2.385374755859374,1.2932531738281259)--(-2.708453369140624,1.4435827636718759)--(-3.1038671874999992,1.6304901123046884)--(-3.6004492187499992,1.868370361328126)--(-4.244561767578125,2.180384521484376)--(-5.11582275390625,2.6063214111328135)--(-6.3634490966796875,3.2207720947265632)--(-8.303624877929689,4.1818112182617195)--(-11.74379302978516,5.893146057128907)--(-15.92,7.976716443983077)--(-15.92,7.976716443983077)--(-15.92,-7.976716547341743)--(-15.92,-7.976716547341743)--(-11.74379302978516,-5.893145751953123)--(-8.303624877929689,-4.181811523437498)--(-6.3634490966796875,-3.2207714843749984)--(-5.11582275390625,-2.6063220214843734)--(-4.244561767578125,-2.1803845214843736)--(-3.6004492187499992,-1.8683703613281235)--(-3.1038671874999992,-1.6304895019531236)--(-2.708453369140624,-1.4435827636718737)--(-2.385374755859374,-1.2932531738281237)--(-2.1157507324218736,-1.1700854492187487)--(-1.8866992187499985,-1.0676647949218736)--(-1.6891223144531236,-0.9814703369140613)--(-1.5164135742187483,-0.9082275390624988)--(-1.3636474609374982,-0.8455078124999987)--(-1.2270788574218732,-0.7914739990234363)--(-1.1038061523437481,-0.7447131347656237)--(-0.9915380859374981,-0.7041217041015613)--(-0.8884387207031231,-0.6688281249999988)--(-0.7930187988281231,-0.6381378173828113)--(-0.704046630859373,-0.6114904785156238)--(-0.620498046874998,-0.5884338378906238)--(-0.541503906249998,-0.5686004638671863)--(-0.46631835937499794,-0.5516912841796863)--(-0.3942932128906229,-0.5374633789062487)--(-0.3248596191406229,-0.5257220458984363)--(-0.25750244140624784,-0.5163110351562488)--(-0.19175903320312282,-0.5091101074218738)--(-0.12719604492187278,-0.5040283203124988)--(-0.06340576171874775,-0.5010040283203113)--(0.,-0.5)--(0.0634057617187523,-0.5010040283203113)--(0.12719604492187733,-0.5040283203124988)--(0.19175903320312734,-0.5091101074218738)--(0.2575024414062524,-0.5163110351562488)--(0.3248596191406274,-0.5257220458984363)--(0.39429321289062746,-0.5374633789062487)--(0.4663183593750025,-0.5516912841796863)--(0.5415039062500026,-0.5686004638671863)--(0.6204980468750025,-0.5884338378906238)--(0.7040466308593776,-0.6114904785156238)--(0.7930187988281276,-0.6381378173828113)--(0.8884387207031277,-0.6688281249999988)--(0.9915380859375027,-0.7041217041015613)--(1.1038061523437528,-0.7447131347656237)--(1.2270788574218778,-0.7914739990234363)--(1.3636474609375029,-0.8455078124999987)--(1.516413574218753,-0.9082275390624988)--(1.689122314453128,-0.9814703369140613)--(1.8866992187500031,-1.0676647949218736)--(2.115750732421878,-1.1700854492187487)--(2.3853747558593783,-1.2932531738281237)--(2.7084533691406283,-1.4435827636718737)--(3.1038671875000037,-1.6304895019531236)--(3.6004492187500037,-1.8683703613281235)--(4.244561767578129,-2.1803845214843736)--(5.115822753906254,-2.6063220214843734)--(6.363449707031255,-3.2207714843749984)--(8.303625488281256,-4.181811523437498)--(11.743793945312508,-5.893145751953123)--(17.008930967255242,-8.52)--(18.16,-8.52)--(18.16,8.78);
\draw [line width=0.7pt,color=ffqqqq,domain=-15.92:18.16] plot(\x,{(-0.-1.*\x)/-5.});
\begin{scriptsize}
\draw[color=ffqqqq] (-15.64,-2.81) node {$eq3$};
\end{scriptsize}

\end{tikzpicture}

\caption{\small Let $f(x,y) = -x^2 + 4 y^2 - 1$ and $g(x,y) = x-5y$. The level set $\{g=0\}$ separates $\{f=0\}$ while $\{g=0\}$ does not separate $\{f<0\}.$}
\label{fig:sep-intro-level-not-sublevel}
\end{minipage}
\end{figure}

In \cite{Quang-Sheu18}, Nguyen and Sheu find that the hypersurface $\{g=0\}$ separates the sublevel set $\{f<0\}$ only when $g$ reduces to an affine function. Here, we are going to study when the hypersurface $\{g=0\}$ separates another hypersurface $\{f=0\}$. First of all, when $\{g=0\}$ separates $\{f=0\}$, the function $g$ needs not to be affine. The following is an example.

\begin{example}
Let $f(x)=x^2-1$ and $g(x)=f(x-1)=x^2-2x$. One can directly see that the level set $\{f=0\}$ is $\{-1,1\}$ and $g(-1)g(1)<0$. By definition, $\{g=0\}$ separates $\{f=0\}$.\end{example}

However, we find that the separation between two quadratic hypersurfaces can actually be reduced to the separation of a hyperplane and a quadratic hypersurface (see Theorem \ref{thm:separation_reduce_linear}). Moreover, if $h$ is an affine function and $\{h=0\}$ separates $\{f=0\}$, the forms of both $f$ and $h$ and the relation between them are all determined, see Theorem   .
With these characterizations, we may establish a non-trivial property for the separation of two quadratic level sets: \textit{mutual separation}, that is, not only $\{g=0\}$ separates $\{f=0\}$ but also $\{g=0\}$ separates $\{f=0\}$. The following is an example of mutual separation with configuration.

\begin{example} \label{ex:mutual-sep}
Let $f(x,y) = -x^2+y^2+1$ and $g(x,y) = -x^2+y^2+2x+1$. Both the 0-level sets $\{f=0\}$ and $\{g=0\}$ are hyperbola. On one hand, left branch of $\{f=0\}$ lies in the region $\{g<0\}$ while right branch lies in the region $\{g>0\}$ (See Fig. \ref{fig:ex_mutual-sep-1}). Hence, $\{g=0\}$ separates $\{f=0\}$. On the other hand, we can also see from Figure \ref{fig:ex_mutual-sep-2} that $\{f=0\}$ separates $\{g=0\}$. Therefore, mutual separation happens for 0-level sets $\{g=0\}$ and $\{f=0\}$.

\begin{figure}
\centering
\begin{subfigure}[c]{0.49\textwidth}
	\centering
	\definecolor{qqffqq}{rgb}{0.,1.,0.}
	\definecolor{wwffqq}{rgb}{0.4,1.,0.}
	\definecolor{ccffww}{rgb}{0.8,1.,0.4}
	\definecolor{ffqqqq}{rgb}{1.,0.,0.}
	\definecolor{qqqqff}{rgb}{0.,0.,1.}
	\begin{tikzpicture}[scale=0.55]
	\clip(-6.1,-5) rectangle (6.36,5);
	\draw [samples=50,domain=-0.99:0.99,rotate around={0.:(0.,0.)},xshift=0.cm,yshift=0.cm,line width=1.pt,color=qqqqff] plot ({1.*(1+(\x)^2)/(1-(\x)^2)},{1.*2*(\x)/(1-(\x)^2)});
	\draw [samples=50,domain=-0.99:0.99,rotate around={0.:(0.,0.)},xshift=0.cm,yshift=0.cm,line width=1.pt,color=qqqqff] plot ({1.*(-1-(\x)^2)/(1-(\x)^2)},{1.*(-2)*(\x)/(1-(\x)^2)});
	\draw [samples=50,domain=-0.99:0.99,rotate around={0.:(1.,0.)},xshift=1.cm,yshift=0.cm,line width=1.pt,dash pattern=on 3pt off 3pt,color=ffqqqq] plot ({1.4142135623730951*(1+(\x)^2)/(1-(\x)^2)},{1.4142135623730951*2*(\x)/(1-(\x)^2)});
	\draw [samples=50,domain=-0.99:0.99,rotate around={0.:(1.,0.)},xshift=1.cm,yshift=0.cm,line width=1.pt,dash pattern=on 3pt off 3pt,color=ffqqqq] plot ({1.4142135623730951*(-1-(\x)^2)/(1-(\x)^2)},{1.4142135623730951*(-2)*(\x)/(1-(\x)^2)});
	\draw[line width=0.8pt,dash pattern=on 2pt off 2pt,color=ccffww,fill=ccffww,pattern=north east lines,pattern color=ccffww](-9.65,5.34)--(-9.65,-5.24)--(-7.660365447998043,-5.24)--(-6.271771030426022,-5.24)--(-5.249139785766599,-5.24)--(-4.465921783447263,-5.24)--(-3.8479508972167946,-4.615477294921873)--(-3.348886108398436,-4.088130493164061)--(-2.9382861328124985,-3.6483007812499983)--(-2.595332641601561,-3.2751208496093738)--(-2.3053097534179674,-2.953823242187499)--(-2.0575241088867178,-2.6736590576171864)--(-1.8440228271484365,-2.426616210937499)--(-1.658770751953124,-2.2065948486328115)--(-1.4971084594726554,-2.0088684082031243)--(-1.3553860473632804,-1.8297113037109367)--(-1.230705261230468,-1.6661468505859367)--(-1.1207403564453118,-1.5157635498046869)--(-1.0236062622070308,-1.3765838623046869)--(-0.9377624511718745,-1.2469659423828119)--(-0.8619421386718745,-1.1255346679687495)--(-0.7950958251953121,-1.0111230468749997)--(-0.7363537597656246,-0.9027319335937496)--(-0.6849914550781246,-0.7994976806640621)--(-0.6404052734374996,-0.7006634521484372)--(-0.6020935058593747,-0.6055609130859373)--(-0.5696423339843747,-0.5135931396484373)--(-0.5427124023437497,-0.4242187499999998)--(-0.5210296630859372,-0.33694335937499986)--(-0.5043786621093748,-0.2513067626953124)--(-0.49259765624999974,-0.16687744140624994)--(-0.48557373046874974,-0.08324157714843747)--(-0.4832397460937497,0.)--(-0.48557373046874974,0.08324157714843747)--(-0.49259765624999974,0.16687744140624994)--(-0.5043786621093748,0.2513067626953124)--(-0.5210296630859372,0.3369430541992186)--(-0.5427124023437497,0.4242187499999998)--(-0.5696423339843747,0.5135931396484373)--(-0.6020935058593747,0.6055612182617185)--(-0.6404052734374996,0.7006634521484372)--(-0.6849914550781246,0.7994976806640621)--(-0.7363537597656246,0.9027322387695309)--(-0.7950958251953121,1.0111230468749997)--(-0.8619421386718745,1.1255343627929684)--(-0.9377624511718745,1.2469659423828119)--(-1.0236062622070308,1.3765835571289056)--(-1.1207403564453118,1.5157635498046869)--(-1.230705261230468,1.6661468505859367)--(-1.3553860473632804,1.8297113037109367)--(-1.4971084594726554,2.0088681030273428)--(-1.658770751953124,2.2065951538085926)--(-1.8440228271484365,2.426616210937499)--(-2.0575241088867178,2.6736593627929675)--(-2.3053097534179674,2.953823242187499)--(-2.595332641601561,3.2751208496093738)--(-2.9382861328124985,3.6483007812499983)--(-3.348886108398436,4.088130569458006)--(-3.8479508972167946,4.615476989746091)--(-4.465921783447263,5.260827059745786)--(-5.249139785766599,5.34)--(-6.271771030426022,5.34)--(-7.660365447998043,5.34)--(-9.65,5.34);
	\draw[line width=0.8pt,dash pattern=on 2pt off 2pt,color=ccffww,fill=ccffww,pattern=north east lines,pattern color=ccffww](11.65,5.34)--(9.660364990234369,5.34)--(8.27177124023437,5.34)--(7.249139404296871,5.34)--(6.465921630859372,5.260827059745786)--(5.847950439453122,4.615476989746091)--(5.348886718749997,4.088130569458006)--(4.938286132812498,3.6483007812499983)--(4.595332031249997,3.2751208496093738)--(4.305310058593747,2.953823242187499)--(4.057524414062498,2.6736593627929675)--(3.8440228271484354,2.426616210937499)--(3.658770751953123,2.2065951538085926)--(3.4971087646484356,2.0088681030273428)--(3.355385742187498,1.8297113037109367)--(3.230704956054686,1.6661468505859367)--(3.1207403564453107,1.5157635498046869)--(3.023606567382811,1.3765835571289056)--(2.9377624511718734,1.2469659423828119)--(2.8619421386718735,1.1255343627929684)--(2.795095825195311,1.0111230468749997)--(2.7363537597656236,0.9027322387695309)--(2.6849914550781233,0.7994976806640621)--(2.6404052734374983,0.7006634521484372)--(2.6020935058593735,0.6055612182617185)--(2.5696423339843735,0.5135931396484373)--(2.5427124023437484,0.4242187499999998)--(2.521029663085936,0.3369430541992186)--(2.5043786621093735,0.2513067626953124)--(2.4925976562499987,0.16687744140624994)--(2.4855737304687486,0.08324157714843747)--(2.4832397460937488,0.)--(2.4855737304687486,-0.08324157714843747)--(2.4925976562499987,-0.16687744140624994)--(2.5043786621093735,-0.2513067626953124)--(2.521029663085936,-0.33694335937499986)--(2.5427124023437484,-0.4242187499999998)--(2.5696423339843735,-0.5135931396484373)--(2.6020935058593735,-0.6055609130859373)--(2.6404052734374983,-0.7006634521484372)--(2.6849914550781233,-0.7994976806640621)--(2.7363537597656236,-0.9027319335937496)--(2.795095825195311,-1.0111230468749997)--(2.8619421386718735,-1.1255346679687495)--(2.9377624511718734,-1.2469659423828119)--(3.023606567382811,-1.3765838623046869)--(3.1207403564453107,-1.5157635498046869)--(3.230704956054686,-1.6661468505859367)--(3.355385742187498,-1.8297113037109367)--(3.4971087646484356,-2.0088684082031243)--(3.658770751953123,-2.2065948486328115)--(3.8440228271484354,-2.426616210937499)--(4.057524414062498,-2.6736590576171864)--(4.305310058593747,-2.953823242187499)--(4.595332031249997,-3.2751208496093738)--(4.938286132812498,-3.6483007812499983)--(5.348886718749997,-4.088130493164061)--(5.847950439453122,-4.615477294921873)--(6.465921630859372,-5.24)--(7.249139404296871,-5.24)--(8.27177124023437,-5.24)--(9.660364990234369,-5.24)--(11.65,-5.24)--(11.65,5.34);
	
	\draw [color=qqqqff](1.28,4.2) node[anchor=north west] {\scriptsize$\{f=0\}$};
	\draw [color=ffqqqq](-3.2,4.7) node[anchor=north west] {\scriptsize$\{g=0\}$};

	\draw [color=red](2.9, 1.2) node[anchor=north west] {\scriptsize$\{g<0\}$};
		\draw [color=red](-4.99, 1.2) node[anchor=north west] {\scriptsize$\{g<0\}$};
	
	\draw [-stealth] (-6.5,0)--(6.35,0);
	\draw [-stealth] (0,-5)--(0,4.98);
	
	\end{tikzpicture}
	
	\caption{The shaded region is sublevel set $\{f<0\}$. The blue solid line is 0-level set $\{f=0\}$ and the red dashed line is 0-level set $\{g=0\}$.}
	\label{fig:ex_mutual-sep-1}
\end{subfigure}
\hfill
\begin{subfigure}[c]{0.49\textwidth}
	\centering

\definecolor{ccffww}{rgb}{0.8,1.,0.4}
\definecolor{qqffqq}{rgb}{0.,1.,0.}
\definecolor{wwffqq}{rgb}{0.4,1.,0.}
\definecolor{ffqqqq}{rgb}{1.,0.,0.}
\definecolor{qqqqff}{rgb}{0.,0.,1.}
\begin{tikzpicture}[scale=0.55]
\clip(-6.1,-5) rectangle (6.36,5);
\draw [samples=50,domain=-0.99:0.99,rotate around={0.:(0.,0.)},xshift=0.cm,yshift=0.cm,line width=1.pt,color=qqqqff] plot ({1.*(1+(\x)^2)/(1-(\x)^2)},{1.*2*(\x)/(1-(\x)^2)});
\draw [samples=50,domain=-0.99:0.99,rotate around={0.:(0.,0.)},xshift=0.cm,yshift=0.cm,line width=1.pt,color=qqqqff] plot ({1.*(-1-(\x)^2)/(1-(\x)^2)},{1.*(-2)*(\x)/(1-(\x)^2)});
\draw [samples=50,domain=-0.99:0.99,rotate around={0.:(1.,0.)},xshift=1.cm,yshift=0.cm,line width=1.pt,dash pattern=on 3pt off 3pt,color=ffqqqq] plot ({1.4142135623730951*(1+(\x)^2)/(1-(\x)^2)},{1.4142135623730951*2*(\x)/(1-(\x)^2)});

\draw [samples=50,domain=-0.99:0.99,rotate around={0.:(1.,0.)},xshift=1.cm,yshift=0.cm,line width=1.pt,dash pattern=on 3pt off 3pt,color=ffqqqq] plot ({1.4142135623730951*(-1-(\x)^2)/(1-(\x)^2)},{1.4142135623730951*(-2)*(\x)/(1-(\x)^2)});

\draw[line width=1.pt,dash pattern=on 2pt off 2pt,color=ccffww,fill=ccffww,pattern=north east lines,pattern color=ccffww](-9.54,5.34)--(-9.54,-5.24)--(-7.375399246215816,-5.24)--(-5.984722166061398,-5.24)--(-5.017272720336911,-4.906427001953123)--(-4.306505126953122,-4.176838989257811)--(-3.763182830810545,-3.6140759277343735)--(-3.335189208984373,-3.165989379882811)--(-2.990048217773436,-2.800069580078124)--(-2.706471557617186,-2.494992675781249)--(-2.4699267578124986,-2.2361883544921866)--(-2.2701571655273423,-2.0133587646484368)--(-2.099719238281249,-1.8190167236328116)--(-1.9530838012695302,-1.6475848388671868)--(-1.8260595703124989,-1.494822387695312)--(-1.7154147338867178,-1.3574414062499993)--(-1.6186199951171867,-1.232854614257812)--(-1.5336700439453117,-1.118992309570312)--(-1.4589575195312492,-1.014177856445312)--(-1.3931814575195305,-0.9170361328124996)--(-1.335281677246093,-0.8264239501953121)--(-1.284386901855468,-0.7413836669921872)--(-1.2397790527343744,-0.6610992431640622)--(-1.2008645629882806,-0.5848724365234372)--(-1.1671511840820306,-0.5120953369140623)--(-1.1382321166992182,-0.4422357177734373)--(-1.113772888183593,-0.37481994628906234)--(-1.0934997558593744,-0.30942199707031237)--(-1.0771932983398431,-0.2456536865234374)--(-1.0646810913085931,-0.18315551757812493)--(-1.0558337402343745,-0.12159240722656245)--(-1.0505606079101557,-0.060643920898437474)--(-1.0488088989257807,0.)--(-1.0505606079101557,0.060643920898437474)--(-1.0558337402343745,0.12159240722656245)--(-1.0646810913085931,0.18315551757812493)--(-1.0771932983398431,0.24565338134765616)--(-1.0934997558593744,0.30942199707031237)--(-1.113772888183593,0.37481994628906234)--(-1.1382321166992182,0.44223602294921854)--(-1.1671511840820306,0.5120956420898435)--(-1.2008645629882806,0.584872131347656)--(-1.2397790527343744,0.6610992431640622)--(-1.284386901855468,0.741383361816406)--(-1.335281677246093,0.8264242553710934)--(-1.3931814575195305,0.9170358276367183)--(-1.4589575195312492,1.014177856445312)--(-1.5336700439453117,1.118992309570312)--(-1.6186199951171867,1.232854614257812)--(-1.7154147338867178,1.3574417114257806)--(-1.8260595703124989,1.4948220825195306)--(-1.9530838012695302,1.6475848388671868)--(-2.099719238281249,1.8190164184570305)--(-2.2701571655273423,2.0133587646484368)--(-2.4699267578124986,2.2361883544921866)--(-2.706471557617186,2.494992675781249)--(-2.990048217773436,2.8000692749023424)--(-3.335189208984373,3.1659892272949204)--(-3.763182830810545,3.6140759277343735)--(-4.306505126953122,4.176839218139647)--(-5.017272720336911,4.9064269256591775)--(-5.984722166061398,5.34)--(-7.375399246215816,5.34)--(-9.54,5.34);
\draw[line width=1.pt,dash pattern=on 2pt off 2pt,color=ccffww,fill=ccffww,pattern=north east lines,pattern color=ccffww](9.54,5.34)--(7.375399169921871,5.34)--(5.984721679687497,5.34)--(5.017272949218747,4.9064269256591775)--(4.306505126953122,4.176839218139647)--(3.763182983398435,3.6140759277343735)--(3.335189208984373,3.1659892272949204)--(2.990048217773436,2.8000692749023424)--(2.706471557617186,2.494992675781249)--(2.4699267578124986,2.2361883544921866)--(2.270156860351561,2.0133587646484368)--(2.099719238281249,1.8190164184570305)--(1.9530834960937489,1.6475848388671868)--(1.8260595703124989,1.4948220825195306)--(1.7154144287109365,1.3574417114257806)--(1.6186199951171867,1.232854614257812)--(1.5336700439453117,1.118992309570312)--(1.4589575195312492,1.014177856445312)--(1.3931817626953118,0.9170358276367183)--(1.3352813720703118,0.8264242553710934)--(1.2843865966796868,0.741383361816406)--(1.2397790527343744,0.6610992431640622)--(1.2008642578124993,0.584872131347656)--(1.1671508789062492,0.5120956420898435)--(1.1382324218749993,0.44223602294921854)--(1.1137731933593744,0.37481994628906234)--(1.0934997558593744,0.30942199707031237)--(1.0771936035156244,0.24565338134765616)--(1.0646813964843744,0.18315551757812493)--(1.0558337402343745,0.12159240722656245)--(1.0505609130859368,0.060643920898437474)--(1.0488085937499994,0.)--(1.0505609130859368,-0.060643920898437474)--(1.0558337402343745,-0.12159240722656245)--(1.0646813964843744,-0.18315551757812493)--(1.0771936035156244,-0.2456536865234374)--(1.0934997558593744,-0.30942199707031237)--(1.1137731933593744,-0.37481994628906234)--(1.1382324218749993,-0.4422357177734373)--(1.1671508789062492,-0.5120953369140623)--(1.2008642578124993,-0.5848724365234372)--(1.2397790527343744,-0.6610992431640622)--(1.2843865966796868,-0.7413836669921872)--(1.3352813720703118,-0.8264239501953121)--(1.3931817626953118,-0.9170361328124996)--(1.4589575195312492,-1.014177856445312)--(1.5336700439453117,-1.118992309570312)--(1.6186199951171867,-1.232854614257812)--(1.7154144287109365,-1.3574414062499993)--(1.8260595703124989,-1.494822387695312)--(1.9530834960937489,-1.6475848388671868)--(2.099719238281249,-1.8190167236328116)--(2.270156860351561,-2.0133587646484368)--(2.4699267578124986,-2.2361883544921866)--(2.706471557617186,-2.494992675781249)--(2.990048217773436,-2.800069580078124)--(3.335189208984373,-3.165989379882811)--(3.763182983398435,-3.6140759277343735)--(4.306505126953122,-4.176838989257811)--(5.017272949218747,-4.906427001953123)--(5.984721679687497,-5.24)--(7.375399169921871,-5.24)--(9.54,-5.24)--(9.54,5.34);

	\draw [color=qqqqff](1.28,4.2) node[anchor=north west] {\scriptsize$\{f=0\}$};
\draw [color=ffqqqq](-3.2,4.7) node[anchor=north west] {\scriptsize$\{g=0\}$};

\draw [color=blue](2.9, 1.2) node[anchor=north west] {\scriptsize$\{f<0\}$};
\draw [color=blue](-4.99, 1.2) node[anchor=north west] {\scriptsize$\{f<0\}$};

\draw [-stealth] (-6.5,0)--(6.35,0);
\draw [-stealth] (0,-5)--(0,4.98);

%

\end{tikzpicture}

	\caption{The shaded region is sublevel set $\{g<0\}$. The blue solid line is 0-level set $\{f=0\}$ and the red dashed line is 0-level set $\{g=0\}$.}
	\label{fig:ex_mutual-sep-2}
\end{subfigure}
\caption{Graph corresponds to Example \ref{ex:mutual-sep}.}
\label{fig:ex_mutual-sep}
\end{figure}
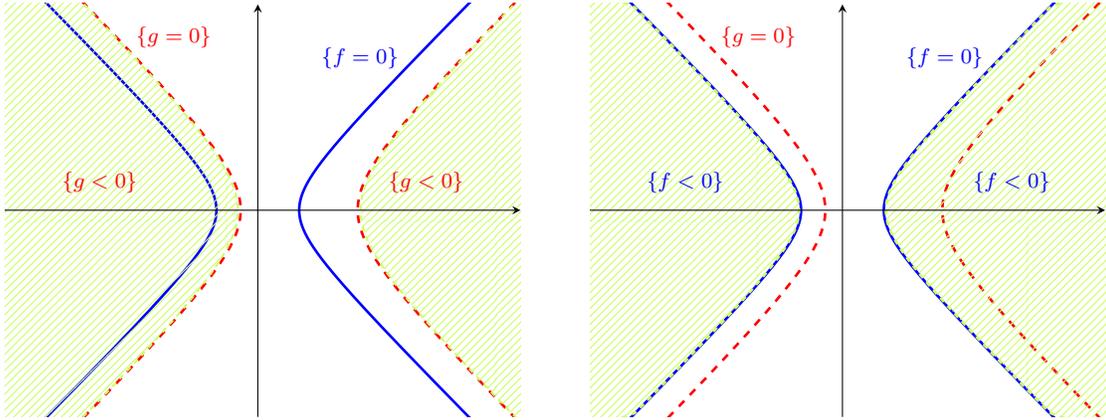
\end{example}

Next, we point out a simple property followed from the separation between two 0-level sets. We can observe that
$\{f = 0\} = \{ \gamma f = 0 \}$ for any $\gamma \neq 0$ and that
$$(\alpha f + \beta g)(u)(\alpha f + \beta g)(v) = \beta^2 g(u)g(v) ~ \text{ for all } u \in L^+, v\in L^-.$$
Hence, if $\{g=0\}$ separates $\{f=0\}$, we have $(\alpha f+\beta g)(u)(\alpha f+\beta g)(v)<0$ for all $u\in\{\gamma f=0\}^+, v\in \{\gamma f=0\}^-$. Therefore we have following proposition.

\begin{proposition}\label{prop:linear_comb_preserve_separation}
If $\{g=0\}$ separates $\{f=0\}$, then $\{\alpha f+\beta g=0\}$ separates $\{\gamma f=0\}$ for all $\alpha, \beta \in \mathbb{R}$ with $\beta \neq 0$, $\gamma \neq 0$.
\end{proposition}

As shown in \cite{Quang-Sheu18}, for a non-constant quadratic function $f(x)$ defined on $\mathbb{R}^n$, after the appropriate change of variables, we may assume that $f(x)$ adopts one of the following five canonical forms:
\begin{eqnarray}
&& -x_1^2-\cdots - x_k^2+\delta (x_{k+1}^2+\cdots +
x_m^2)+\theta; \label{form:1}\\
&&-x_1^2-\cdots - x_k^2+\delta(x_{k+1}^2+\cdots + x_m^2)-1;\label{form:2}\\
&&-x_1^2-\cdots - x_k^2+\delta (x_{k+1}^2+\cdots+ x_m^2)+x_{m+1};\label{form:3}\\
&&\hskip8pt x_1^2+\cdots + x_m^2+\delta x_{m+1}+c';\label{form:4}\\
&& \hskip8pt \delta x_1+c', \label{form:5}
\end{eqnarray}
where $\delta, {\theta\in \{0, 1\}}$.

Before proceeding, we first quote two results from \cite{Quang-Sheu18}. One is a direct consequence of the Intermediate Value Theorem. The other specifies the form of $f$ adopts when the sublevel sets $\{f < 0\}$ or $\{f \leq 0\}$ have disconnected components.

\begin{lemma}[\cite{Quang-Sheu18}]\label{lem:IVT}
Let $C$ be a connected set in $\mathbb{R}^n$, $h$ be continuous on $C$ and \begin{equation}\label{e001}
C\cap \{h=\alpha_0\}=\emptyset
\end{equation} then $(h(x)-\alpha_0) \cdot (h(y)-\alpha_0)>0~ \forall x, y \in C.$
\end{lemma}

\begin{lemma}[\cite{Quang-Sheu18}]\label{lem:disconnected_f<0}
If $f(x)$ is  a non-constant quadratic function then the following results characterize other sublevel sets of a quadratic function $f(x)$.
\begin{itemize}
\item[\rm(i)] The sublevel set $\{f <0\}$ has at most two connected components. It contains exactly two connected components if and only if $f$ has form \eqref{form:1} with $k=1$. Furthermore, the two connected components are
\begin{equation}\label{two-components}
\begin{aligned}
\{f <0\}^- &= \{x\in \{f <0\} |~ x_1<0\} \\
\{f <0\}^+ &= \{x\in \{f <0\} |~ x_1>0\}.
\end{aligned}
\end{equation}
\item[\rm(ii)] $\{f \leq 0\}$ has exactly two connected components if and only if $f$ is of form \eqref{form:1} with $k=1, \theta=1$.	
\end{itemize}
\end{lemma}

\begin{remark}\label{rmk:f_form_f=0} According to Lemma \eqref{lem:disconnected_f<0}, if the set $\{f \leq 0\}$ is disconnected,
the quadratic function $f$ must be of form (\ref{form:1}) with $k=1,\theta=1$. Then, $-f$ cannot, again, adopt (\ref{form:1}) with $k=1,\theta=1.$ Hence, when $\{f \leq 0\}$ is disconnected, $ \{-f \leq 0\}=\{f\ge0\}$ must be connected.
\end{remark}

In the next section, we study the property of disconnected 0-level set $\{f=0\}$ and the separation between two 0-level sets.

\section{Separation Theory for Level Set} \label{sec:SeparationTheory}
The following Lemma \ref{lem:disconnected_f=0} shows that the disconnectedness of $\{f = 0\}$ can be reflected in the disconnectedness of $\{f \leq 0\}$ or $\{f \geq 0\}$.

\begin{lemma} \label{lem:disconnected_f=0} Let  $f(x)$ be a non-constant quadratic function. Then $\{f = 0\}$ is disconnected if and only if either $\{f \leq 0\}$ or  $\{f \geq 0\}(= \{-f \leq 0\})$ is disconnected. When $\{f = 0\}$ is disconnected, it has exactly two connected components.
\end{lemma}
\begin{proof}
[Proof for sufficiency] If $\{f \leq 0\}$ or $\{-f \leq 0\}$ is disconnected, by Lemma \ref{lem:disconnected_f<0}, $f(x)$ or $-f(x)$ has the form: $-x_1^2+\delta(x_2^2+\cdots +x_m^2)+1$. To verify that $\{f=0\}$ is disconnected when $f$ adopting these forms, we can use the same argument of checking the disconnectedness of $\{f<0\}$ in Lemma 1 in \cite{Quang-Sheu18}.

\vskip 0.25cm
\noindent{[Proof for necessity]} We first claim that after some linear transformations to put $f(x)$ to canonical form, $f(x)$ cannot have non-vanished linear terms; otherwise, $\{f = 0\}$ will be connected. If the canonical form of $f(x)$ had non-vanished linear terms, by our classification in the beginning of Section \ref{sec:SeparationTheory}, it had form (\ref{form:3}), form (\ref{form:4}) with $\delta\ne 0$, or form (\ref{form:5}) with $\delta\ne 0$. It means $f(x)$ has form $\rho(-x_1^2-\cdots-x_k^2)+\rho'(x_{k+1}^2+\cdots+x_{m}^2)+x_{m+1}$ where $\rho, \rho'\in\{0, 1\}$. Obviously, any point $u=(u_1, \cdots, u_n)\in \{f=0\}$ can connect to  $0\in \{f = 0\}$ by the curve $\gamma(t)=(tu_1, \cdots, tu_m, t^2u_{m+1}, \cdots, t^2u_n)~\text{with } t \in [0,1]$ in $\{f=0\}$. This establishs the claim.

Now, we know $f(x)=\rho(-x_1^2-\cdots -x_k^2)+\rho'(x_{k+1}^2+\cdots+x_{m}^2)+\gamma$, where $\rho,\rho'\in\{0, 1\}$ ($f(x)$ has no linear terms). Without loss of generality, when $\gamma>0$, we choose $\gamma=1$; and choose $\gamma=-1$ for $\gamma<0$. We shall see that the choice will not affect any topological properties of the level sets of $f(x)$.
Let $\lambda_{-}$ and $\lambda_{+}$ be the number of negative and positive eigenvalues of $A$ respectively. The following discussion is divided into two cases: either ($\lambda_{-}\geq 2$ and $\lambda_{+}\geq 2$) or ($\lambda_{-}\le 1$ or $\lambda_{+}\le 1$).

\noindent $\bullet$ For ($\lambda_{-}\geq 2$ and $\lambda_{+}\geq 2$):

In this case, $k \ge 2$ and $m \ge k+2$. Since the number of negative and positive eigenvalues are both larger than $2$, we can use the rotation argument described in the proof of Lemma \ref{lem:disconnected_f<0} in \cite{Quang-Sheu18} to construct a ``piecewise arc curve'', which connects any given point $u=(u_1, \cdots, u_n)^T$ in $\{f = 0\}$ to a point $\bar{u}=(u_1^*, 0, \cdots, 0, u_{k+1}^*, 0, \cdots, 0)^T$ in the intersection of $\{f = 0\}$ and the plane spanned by the first coordinate $x_1$ and the $(k+1)^{th}$ coordinate $x_{k+1}$, where $u_1^*=\sqrt{u_1^2+\cdots +u_k^2}$ and $u_{k+1}^*=\sqrt{u_{k+1}^2+\cdots +u_m^2}$ satisfying $-(u^*_1)^2+(u_{k+1}^*)^2+\gamma=0$.

If we view on this plane, we can find that the point $\bar{u}$ always lie on a half of a piece of hyperbola (when $\gamma \neq 0$) or a $45$-degree ray (when $\gamma = 0$) in the first quadrant. Therefore, after linking the given point $u$ to $\bar{u}$, we can further trace the hyperbola or the straight line to link $\bar{u}$ to the axis ($x_1$ or $x_{k+1}$). Which axis we should link to depends on the sign of the constant $\gamma$:

-When $\gamma$ is negative, say $\gamma=-1$, the curve $$\Gamma_-(t)=(t, 0, \cdots, 0, \sqrt{t^2+1},0, \cdots,  0)^T,\ \ \text{where } t \in [0,u_{1}^{\ast}]$$ can connects $\bar{u}=\Gamma_-(u^*_1)$ to the point $\Gamma_-(0)=(0, \cdots, 0, 1, 0, \cdots, 0)^T$ on the $(k+1)^{th}$ axis. The square root $\sqrt{t^2+1}$ in $\Gamma_-(t)$ sits in the $(k+1)^{th}$ component.

-When $\gamma$ is positive, say $\gamma=1$, the curve $$\Gamma_+(t)=(\sqrt{t^2+\gamma}, 0, \cdots, 0, t, 0, \cdots,  0)^T,\ \ \text{where } t \in [0,u_{k+1}^{\ast}] $$ can connects $\bar{u}=\Gamma_+(u^*_{k+1})$ to the point $\Gamma_+(0)=(1, \cdots, 0, \cdots, 0)^T$ on the first axis.

-When $\gamma = 0$, note that $(u^*_1)^2 = (u_{k+1}^*)^2$ and the curve $$\Gamma_0(t)=(t, 0, \cdots, 0, t, 0, \cdots,  0)^T,\ \ \text{where } t \in [0,u_{1}^{\ast}] $$ can connects $\bar{u}=\Gamma_0(u^*_{1})$ to the origin $\Gamma_0(0)=(0, \cdots, 0)^T$.

\noindent Since any point $u$ in $\{f = 0\}$ can be first connected to $\bar{u}$ and hence to $\Gamma_-(0)$ or $\Gamma_+(0)$ or the origin, $\{f = 0\}$ is connected. In summary, for the case ($\lambda_{-}\geq 2$ and $\lambda_{+}\geq 2$), $\{f = 0\}$ must be connected.

\noindent $\bullet$  For ($\lambda_{-}\le1$ or $\lambda_{+}\le1$):

\noindent\emph{Case 1:} $\lambda_{-}=0$. Namely,
$f(x)=x_1^2+\cdots +x_m^2+\gamma$. When $\gamma \geq 0$,
$\{f = 0\}$ is connected since $\{f = 0\}$ is either an empty set or an $(n-m)$-dimensional subspace of $\mathbb{R}^n$. Therefore, we just need to consider the case $\gamma = -1$. When $\gamma=-1$ and $m\geq 2$, that is, the number of positive eigenvalues is larger than $2$, we can use the rotation argument again to connect any point in $\{f = 0\}$ to a fixed point $(1,0,\cdots,0)$ on the first axis. Hence, $\{f = 0\}$ is connected. When
$\gamma=-1$ and $m=1$, $\{f = 0\}=\{x\in \mathbb{R}^n |~ x_1^2=1
\}$ is disconnected. In this case, $-f(x)$ is of form (\ref{form:1}) with $k=1$ and $\theta=1$, and Lemma \ref{lem:disconnected_f<0} (ii) implies that $\{-f \leq 0\}=\{f \geq 0\}$ is disconnected. In summary, $\{f = 0\}$ is disconnected for the case $\lambda_{-}=0$ if and only if $\{-f \leq 0\}=\{f \geq 0\}$ is disconnected.

\noindent\emph{Case 2:} $\lambda_-=1$. In this case, $f(x) = -x_1^2+\delta(x_2^2\cdots+x_m^2)+\gamma=0$ with $m \geq 2$.

\noindent - When $\gamma=0$, the level set $\{f=0\}$ can be rewritten as $\{x_1^2 = x_2^2\cdots+x_m^2\}$. We can link any point $x$ to the origin by the line $L(t) = tx$ with $t \in [0,1]$. Hence, $\{f = 0\}$ is connected. In this case, one can see that the cross section of $\{f=0\}$ with cutting plane $\{x_1=k\}$ is an $(m-1)$-dimensional sphere with radius $\sqrt{|k|}$. For example, if $n=3$, the level set $\{f=0\}$ consists of two circular cones with vertices at origin. For any $x$ in $\{f = 0\}$, the whole line segment between $x$ and the origin lies in $\{f = 0\}$. We link every point to the origin via this segment exactly.

\noindent - When $\gamma =1$, $f(x) = -x_1^2 + \delta(x_2^2\cdots+x_m^2)+1$ is of hyperboloid-type, and the level set $\{f = 0\}=\{x\in \mathbb{R}^n |~x_1=\pm\sqrt{1+\delta(x_2^2\cdots+x_m^2)}\}$ is certainly disconnected. In this case, $f$ has form (\ref{form:1}) with
$k=1, \theta=1$. By Lemma \ref{lem:disconnected_f<0} $(i)$, $\{f \leq 0\}$ is
disconnected too.

\noindent - When $\gamma =-1$, we have $\{f = 0\}=\{x\in \mathbb{R}^n |~
-x_1^2+\delta(x_2^2\cdots+x_m^2)=1\}$.

If $\delta =1$ and $m\geq 3$, by the rotation argument again, any point $u=(u_1, \cdots u_n)^T$ in $\{f = 0\}$ can be connected to $\bar{u}=(u_1,\sqrt{u_2^2+\cdots+u_m^2},0,\cdots, 0)^T$, a point on the plane spanned by the first and second coordinates ($x_1$ and $x_2$). Similar to the case of ($\lambda_{-}\geq 2$ and $\lambda_{+}\geq 2$), in the view of this plane, any $\bar{u}$ lies on one branch of the hyperbola. Therefore, we can also link $\bar{u}$ to a fixed point $(1,0,\cdots,0)$ on the axis $x_1$ by tracing this branch of hyperbola. More precisely, we link $\bar{u}$ to the axis by the curve $\Gamma(t)=(t, \sqrt{1+t^2}, 0, \cdots,
0)^T$, where $t$ lies between $0$ and $u_1$.

If $\delta =1$ and $m=2$, we have $f(x)=-x_1^2+x_2^2-1$. The level set $\{f = 0\}=\{x\in \mathbb{R}^n | x_2=\pm\sqrt{1+x_1^2}\}$ is disconnected and $-f$ is of form (\ref{form:1}) with $k=1, \theta=1$. By Lemma \ref{lem:disconnected_f<0}(ii), $\{-f \leq 0\}=\{f \geq 0\}$ is disconnected, too.

If $\delta=0$, then $\{f = 0\}=\{x\in \mathbb{R}^n |~
-x_1^2-1=0\}=\emptyset$.

\vskip 0.15cm
\noindent In summary, for $\lambda_{-}\le1$, $\{f = 0\}$ is disconnected
when and only when:
\begin{itemize}
	\item[•] $\lambda_{-}=0,~\gamma =-1,~m=1,$ in which case
	$\{f \geq0\}$ is disconnected;
	\item[•] $\lambda_{-}=1,~\gamma=1,$
	in which case $\{f \leq 0\}$ is disconnected;
	\item[•] $\lambda_{-}=1,~\gamma=-1,~\delta=1,~m=2,$ in which case
	$\{f \geq 0\}$ is disconnected.
\end{itemize}

Finally, the cases $\lambda_+=0 $ (or $\lambda_+=1$) is
equivalent to the cases when the number $\lambda_-$ of  $-f$ is
$0$ (or $1$ respectively) and $\{-f = 0\}=\{f = 0\}$. Therefore, when
$\lambda_+\le1,$ $\{f = 0\}$ is disconnected if and only if
either $f$ or $-f$ has form (\ref{form:1}) with $k=1, \theta=1,$
namely, either $\{f \leq0\}$ or  $\{f \geq 0\}$ is
disconnected.
\end{proof}

Before we study the behaviors of two 0-level sets, we first give an useful alternative definition for the separation between two 0-level sets as follows.

\begin{proposition}\label{prop:criterion}
The 0-level set $\{g=0\}$ separates the 0-level set $\{f=0\}$ if and only if
\begin{align}
&\{f=0\}\cap\{g=0\}=\emptyset ~ \text{ and } \label{eq:alt_def-1} \\
&g(u)g(v)<0 ~ \text{ for some }~u, v \in\{f=0\}. \label{eq:alt_def-2}
\end{align}
\end{proposition}
\begin{proof} The necessity is obvious, so it suffices to prove the sufficiency. Let $L^+ = \{f=0\} \cap \{g>0\}$ and $L^- = \{f=0\} \cap \{g<0\}$. Since $\{f=0\}\cap\{g=0\}=\emptyset$, we have $\{f=0\} = L^+ \cup L^-$. Both $L^+$ and $L^-$ are non-empty due to the assumption: there exist $u, v \in\{f=0\}$ such that $g(u)g(v)<0$. Finally, the definition of $L^+$ and $L^-$ directly implies that $g(u)g(v)<0$ for all $u \in L^+$ and $v \in L^-$.
\end{proof}

As we said in Section \ref{sec:Intro}, it is possible that a quadratic hypersurface separates another quadratic hypersurface. Nonetheless, Theorem \ref{thm:separation_reduce_linear} shows that this can be reduced to the case that a hyperplane separates a quadratic hypersurface.

\begin{theorem} \label{thm:separation_reduce_linear}
The set $\{g=0\}$ separates the set $\{f=0\}$  if and only if $B=\lambda A$ for some $\lambda\in\mathbb{R}$ and $\{-\lambda f+g=0\}$ separates $\{f=0\}$.
\end{theorem}
\begin{proof}{[ Proof for necessity]}

Since $\{f=0\}$ must be disconnected, by Lemma \ref{lem:disconnected_f<0} and Lemma \ref{lem:disconnected_f=0}, there exists a basis of $\mathbb{R}^n$ such that $f(x)$ or $-f(x)$ has the canonical form:
\[-x_1^2+\delta(x_2^2+\cdots+x_m^2)+1\]
and \[\{f=0\}=\{f=0\}^+\cup \{f=0\}^-, \text{ where } \{f=0\}^{\pm}=\{x\in  \{f=0\}  | \pm x_1>0\}.\]

With above new basis, we write $g(x)=x^TBx+2b^Tx+b_0$. By Lemma \ref{lem:IVT}, $g(x)$ has the same sign for all $x$ belongs to one connected component in $\{f=0\}$. Without loss of generality, assume that
\begin{equation}\label{mm1}
g(x)<0~\forall x\in \{f=0\}^- \text{ and } g(x)>0~ \forall~x\in \{f=0\}^+.\end{equation}

\noindent{Now, we first consider the case $\delta=1$.}

- Step 1 :
It is easy to see that $\pm 1e_1+te_i\in\{f=0\}, ~\forall t\in \mathbb{R}, i >  m.$ Then \eqref{mm1} implies that

$
\begin{cases}
g(-1e_1+te_i)<0 & \forall t\\
g(+1e_1+te_i)>0 & \forall t\\
\end{cases}
\Rightarrow \begin{cases}
b_{11}+b_{ii}t^2-2b_{1i}t -2b_1+2b_it+b_0<0& \forall t\\
b_{11}+b_{ii}t^2+2b_{1i}t +2b_1+2b_it+b_0>0 & \forall t\\
\end{cases}$
In the view point of polynomials in $t$, the above two inequalities show that one of the polynomials of degree $2$ is always positive and the other is always negative. One can directly derive that
\begin{equation}
b_{ii}=b_{1i}=b_{i1}=b_i=0, ~\forall i>m. \label{mm18}
\end{equation}

- Step 2 :
We have $\sqrt{2}e_1 + e_2+te_i\in\{f=0\}^+, ~\forall t\in \mathbb{R}, i > m$. Then \eqref{mm1} implies that $g(\sqrt{2}e_1 + e_2+te_i)>0 ~, \forall t \in \mathbb{R}$. Combined with \eqref{mm18}, one has
$$2b_{11}+b_{22}+2\sqrt{2}b_{12}+2b_{2i}t+\sqrt{2}b_1+2b_2+b_0>0 ~~~ \forall t$$
It implies that $b_{2i}=0, ~\forall i>m$. By the same argument, one can derive $b_{3i}=\cdots = b_{ni}=0, \forall i>m$. Since the matrix $B$ is symmetric, we therefore have
\begin{equation}
b_{2i}=b_{i2}=b_{3i}=b_{i3}=\cdots = b_{ni}=b_{in}=0, \forall i>m.\label{mm19}
\end{equation}

- Step 3 :
We have
\begin{eqnarray*}
te_1\pm \sqrt{t^2-1}e_2\in\{f=0\}^+ ~\forall t>1,\\
-te_1\pm \sqrt{t^2-1}e_2\in\{f=0\}^- ~\forall t>1.
\end{eqnarray*}
Again, \eqref{mm1} implies that  	
\begin{eqnarray}
b_{11}t^2+b_{22}(t^2-1)+2b_{12}t\sqrt{t^2-1}+2b_1t+2b_2\sqrt{t^2-1}+b_0>0 ~~\forall t>1 ~~\label{m18} \\
b_{11}t^2+b_{22}(t^2-1)-2b_{12}t\sqrt{t^2-1}+2b_1t-2b_2\sqrt{t^2-1}+b_0>0 ~~\forall t>1 ~~\label{m19}\\
b_{11}t^2+b_{22}(t^2-1)-2b_{12}t\sqrt{t^2-1}-2b_1t+2b_2\sqrt{t^2-1}+b_0<0 ~~\forall t>1 ~~\label{m20}\\
b_{11}t^2+b_{22}(t^2-1)+2b_{12}t\sqrt{t^2-1}-2b_1t-2b_2\sqrt{t^2-1}+b_0<0 ~~\forall t>1 ~~\label{m21}
\end{eqnarray}	
Consider the following four combinations:
\begin{align*}
\eqref{m18}+\eqref{m19}&:~2(b_{11}+b_{22})t^2+4b_1t+2(b_0-b_{22})>0~,~ \forall t>1\\
\eqref{m20}+\eqref{m21}&:~2(b_{11}+b_{22})t^2-4b_1t+2(b_0-b_{22})<0~,~ \forall t>1\\
\eqref{m18}-\eqref{m20}&:~4b_{12}t \sqrt{t^2-1}+4b_1t>0~,~\forall t>1\\
\eqref{m19}-\eqref{m21}&:~4b_{12}t \sqrt{t^2-1}-4b_1t<0~,~\forall t>1
\end{align*}
Observing the leading coefficients of these four polynomials in $t$, we have
\begin{align}
b_{11}=-b_{22},\label{mm24}\\
b_{12}=b_{21}=0.\label{mm25}
\end{align}
By the same argument, we further have
\begin{align}
b_{11}=-b_{ii}, i=3, \cdots, m, \label{mm26}\\
b_{1i}=b_{i1}=0, i=3, \cdots, m. \label{mm27}
\end{align}

- Step 4 :
We have \begin{eqnarray*}
	te_1+\dfrac{\sqrt{2}}{2}\sqrt{t^2-1}e_2\pm \dfrac{\sqrt{2}}{2}\sqrt{t^2-1}e_3\in\{f=0\}^+~\forall t>1,\\
	te_1-\dfrac{\sqrt{2}}{2}\sqrt{t^2-1}e_2\pm \dfrac{\sqrt{2}}{2}\sqrt{t^2-1}e_3\in\{f=0\}^+~ \forall t>1.
\end{eqnarray*}
Combined with the results in \eqref{mm24},  \eqref{mm25}, \eqref{mm26} and \eqref{mm27}, statement \eqref{mm1} implies that
\begin{align}
b_{11}t^2-{b_{11}}(t^2-1)+b_{23}(t^2-1)+2b_1t+(+\sqrt{2}b_2+\sqrt{2}b_3)\sqrt{t^2-1}+b_0>0~\forall t>1,\label{m26}\\
b_{11}t^2 -{b_{11}}(t^2-1)-b_{23}(t^2-1)+2b_1t+(+\sqrt{2}b_2-\sqrt{2}b_3)\sqrt{t^2-1}+b_0>0~\forall t>1,\label{m27}\\
b_{11}t^2 -{b_{11}}(t^2-1)-b_{23}(t^2-1)+2b_1t+(-\sqrt{2}b_2+\sqrt{2}b_3)\sqrt{t^2-1}+b_0>0~\forall t>1,\label{m28}\\
b_{11}t^2 -{b_{11}}(t^2-1)+b_{23}(t^2-1)+2b_1t+(-\sqrt{2}b_2-\sqrt{2}b_3)\sqrt{t^2-1}+b_0>0~\forall t>1.\label{m29}
\end{align}
As in Step 3, via observing two inequalities: \eqref{m26}+\eqref{m29} and \eqref{m27}+\eqref{m28}, the result $b_{23}=b_{32}=0$ yields. Analogously, we have
\begin{equation}
b_{ij}=0 ~\forall i\ne j. \label{mm32}
\end{equation}

Combining the results in the above four steps and noticing that the matrix of $f(x)$ is $A=diag(-1,1,\cdots,1)$, we get following conclusion
\begin{eqnarray}\label{vnnv}
B=\lambda A \text{ for some } \lambda \in \mathbb{R}, \text{ and } b_i=0 ~ \forall i>m.
\end{eqnarray}
When $\delta=0$, by the same argument, \eqref{vnnv} also holds with $m=1$. Hence, $B$ must equal to $\lambda A$ for some $ \lambda \in \mathbb{R}$ with above new basis. Therefore, this conclusion still holds for initial basis.

Since $\{g=0\}\cap \{f=0\}$ is empty and $\{-\lambda f+g=0\}\cap \{f=0\} = \{g=0\}\cap \{f=0\}$, we get $\{-\lambda f+g=0\}\cap \{f=0\}$ is empty. Moreover, \eqref{mm1} shows that $(-\lambda f+g)(u)(-\lambda f+g)(v)=g(u)g(v)<0$ for all $u\in \{f=0\}^-, v\in \{f=0\}^+$. By Proposition \ref{prop:criterion}, $\{f=0\}$ is separated by  $\{-\lambda f+g=0\}$.

\noindent{[ Proof for sufficiency]} Suppose that $\{f=0\}$ is separated by  $\{-\lambda f+g=0\}$. This implies that $(-\lambda f+g)(u)(-\lambda f+g)(v)<0~ \forall u\in \{f=0\}^+, v\in \{f=0\}^- $. Since $g(u)g(v)=(-\lambda f+g)(u)(-\lambda f+g)(v)$, we can conclude $\{g=0\}$ separates $\{f=0\}$.
\end{proof}

\begin{remark}
The condition $B = \lambda A$ in Theorem \ref{thm:separation_reduce_linear} ensures that $- \lambda f + g$ becomes an affine function.
\end{remark}

The statement of the following lemma (Lemma \ref{lem:separation_LQ_condition}) is very similar to that of Theorem 1 in \cite{Quang-Sheu18}. In fact, the proof of this lemma is almost the same as the proof of Theorem 1 in \cite{Quang-Sheu18}, so in the following, we will refer to the proof in \cite{Quang-Sheu18} and just point out the differences between these two proofs.

\begin{lemma} \label{lem:separation_LQ_condition}
Suppose that $h(x)$ is a affine function, then $\{h=0\}$ separates $\{f=0\}$ if and only if  there exists a basis such that
	
\begin{description}
\item[$\rm(i)$] $f(x)$  has the form $-x_1^2+\delta(x_2^2+\cdots+x_m^2)+1,~ \delta \in\{0, 1\};$
\item[$\rm(ii)$] With the same basis,  $h(x)$ has the form $c_1x_1+\delta(c_2x_2+\cdots+c_m x_m)+c_0$, where $c_1 \ne 0;$
\item[$\rm(iii)$] $f(x)|_{\{h = 0\}}=-\left( \delta\sum_{i=2}^{m}\dfrac{c_i}{c_1}x_i+\dfrac{c_0}{c_1} \right)^2+\delta\sum_{i=2}^{m}x_i^2+1> 0,~\forall (x_2, \cdots,x_n)^T\in\mathbb{R}^{n-1}.$
\end{description}
or
\begin{description}
\item[$\rm(i)'$] $-f(x)$  has the form $-x_1^2+\delta(x_2^2+\cdots+x_m^2)+1,~ \delta \in\{0, 1\};$
\item[$\rm(ii)'$] With the same basis,  $h(x)$ has the form $c_1x_1+\delta(c_2x_2+\cdots+c_mx_m)+c_0$, where $c_1\ne 0;$
\item[$\rm(iii)'$] $-f(x)|_{\{h = 0\}}=-\left(\delta\sum_{i=2}^{m}\dfrac{c_i}{c_1}x_i+\dfrac{c_0}{c_1}\right)^2+\delta\sum_{i=2}^{m}x_i^2+1> 0,~\forall (x_2, \cdots,x_n)^T\in\mathbb{R}^{n-1}.$
\end{description}
Therefore if $h(x)$ is a affine function then $\{h=0\}$ separates $\{f=0\}$ if and only if either $\{h=0\}$ separates $\{f\leq 0\}$ or $\{h=0\}$ separates $\{-f\leq 0\}$.
\end{lemma}
\begin{proof}
[The necessary part]: Since $\{f=0\}$ is disconnected, by Lemma \ref{lem:disconnected_f=0}, $f(x)$ or $-f(x)$ has the form:  \begin{equation}\label{fff} -x_1^2+\delta(x_2^2+\cdots
+x_m^2)+1,~ \delta \in\{0, 1\}.\end{equation}

We first consider the case $f(x)= -x_1^2+\delta(x_2^2+\cdots+x_m^2)+1,~ \delta \in\{0, 1\}$. Since $h(x)$ is affine function, we may assume $h(x)=c_1x_1+\cdots+c_nx_n+c_0$.

- The linear term of $h$ corresponding to the only negative eigenvalue of $f$ is non-zero. That is $c_1 \neq 0$. In the proof of Theorem $1$ in \cite{Quang-Sheu18}, we have the same conclusion, so we refer to the proof of this part in \cite{Quang-Sheu18}. One can just replace the set $\{f<0\}$ in \cite{Quang-Sheu18} by $\{f=0\}$ and points $s^{\pm}$ in \cite{Quang-Sheu18} by
\begin{align*}
s^{\pm}&=(\pm 1, 0, \cdots, 0, \dfrac{-c_0}{c_i}, 0, \cdots, 0)\in\{f=0\}^{\pm}~~\text{for } i>m ~ ;\\
s^{\pm}&=(\pm \sqrt{1+\delta\dfrac{c_0^2}{c_i^2}}, 0, \cdots, 0, \dfrac{-c_0}{c_i}, 0, \cdots, 0)\in\{f=0\}^{\pm}~~\text{for } i\leq m
\end{align*}
The desired result is derived.


-  The linear terms of $h$ corresponding to 0-eigenvalues of $A$ must be $0$. Namely, $c_j=0, \forall j>m$ and in the case when $\delta=0$ in \eqref{fff}, $c_j=0, \forall j\geq 2$. Again, we refer to the proof of this part in \cite{Quang-Sheu18}. By substituting $\{f=0\}$ for $\{f<0\}$ in the original proof and $w^{\pm}=\pm(1, 0, \cdots, 0, \dfrac{-c_1}{c_i}, 0, \cdots, 0)$ for the original $w^{\pm}$, we get the desired result.


- For the validity of $\rm{(iii)}$: since $\{h=0\}$ separates $\{f=0\}$, we have $\{f=0\}\cap \{h=0\}=\emptyset$, which implies either $\{h=0\} \subset \{f<0\}$ or $\{h=0\} \subset \{f>0\}$. Namely, for $x \in \{h=0\}$ that $x_1=-(\delta\dfrac{c_2}{c_1}x_2+\cdots+\delta\dfrac{c_m}{c_1}x_m+\dfrac{c_0}{c_1})$, there is either
\begin{equation}\label{fff1}
-\left(\delta\sum_{i=2}^{m}\dfrac{c_i}{c_1}x_i+\dfrac{c_0}{c_1}\right)^2+\delta\sum_{i=2}^{m}x_i^2+1> 0,~\forall (x_2, \cdots,
x_n)^T\in\mathbb{R}^{n-1},
\end{equation}
or
\begin{equation}\label{fff2}
-\left(\delta\sum_{i=2}^{m}\dfrac{c_i}{c_1}x_i+\dfrac{c_0}{c_1}\right)^2+\delta\sum_{i=2}^{m}x_i^2+1<0,~\forall (x_2, \cdots,
x_n)^T\in\mathbb{R}^{n-1}.
\end{equation}

To determine which case it is, with $x_2=\cdots=x_n=0$, we only need to determine the sign of $1-\frac{c_0^2}{c_1^2}$. Since $u^{\pm} = (\pm 1,0,\cdots,0) \in \{f=0\}^{\pm}$ and $\{g=0\}$ separates $\{f=0\}$, we have $h(u^+)h(u^-)=-c_1^2(1-\frac{c_0^2}{c_1^2})<0$, and hence $1-\frac{c_0^2}{c_1^2}>1$. Namely, \eqref{fff1} occurs. This prove (iii).

When $-f(x)= -x_1^2+\delta(x_2^2+\cdots+x_m^2)+1,~ \delta \in\{0, 1\}$, $\rm{(i)}'$, $\rm{(ii)}'$, and $\rm{(iii)}'$ can be derived by an analogous argument.

\vskip 0.3cm
\noindent [The sufficient part]: Without loss of generality, we can only prove that $\rm{(i)}$, $\rm{(ii)}$, and $\rm{(iii)}$ imply that $\{h=0\}$ separates $\{f=0\}$. First, $\rm{(i)}$ and $\rm{(ii)}$ imply that both $\{f=0\}$ and $\{h=0\}$ are non-empty. Second, $\rm{(iii)}$ gives us $\{h=0\} \subset \{f>0\}$, and hence $\{f=0\}\cap \{h=0\}=\emptyset$. Third, with $x_2=\cdots=x_n=0$ in $\rm{(iii)}$, we have $1-\frac{c_0^2}{c_1^2}>0$ and therefore $h(u^+)h(u^-)=-c_1^2(1-\frac{c_0^2}{c_1^2})<0$, where $u^{\pm} = (\pm 1,0,\cdots,0) \in \{f=0\}$. By Proposition \ref{prop:criterion}, we have $\{h=0\}$ separates $\{f=0\}$.
\end{proof}

\begin{corollary} \label{cor:chckable}
Suppose that $h(x) = c^T x + c_0$ is a affine function, then $\{h=0\}$ separates $\{f=0\}$ if and only if the following three hold:
\begin{enumerate}[(a)]
\setlength{\itemsep}{2pt}
\item $A$ has exactly one negative eigenvalue, $a \in \mathcal{R}(A)$
\item $c \in \mathcal{R}(A)$, $c \neq 0$
\item $V^T A V \succeq 0$, $w \in \mathcal{R}(V^T A V)$, and $f(x_0)-w^T (V^T A V^T)^{\dagger} w > 0$,
\end{enumerate}
or
\begin{enumerate}
\setlength{\itemsep}{2pt}
\item[(a$'$)] $A$ has exactly one positive eigenvalue, $a \in \mathcal{R}(A)$
\item[(b$'$)] $c \in \mathcal{R}(A)$, $c \neq 0$
\item[(c$'$)] $V^T A V \preceq 0$, $w \in \mathcal{R}(V^T A V)$, and $f(x_0)-w^T (V^T A V^T)^{\dagger} w < 0$,
\end{enumerate}
where $w = V^T(Ax_0+a)$, $x_0 = \frac{-c_0}{c^T c} c$, and $V \in \mathbb{R}^{n \times (n-1)}$ is the matrix basis for $\mathcal{N}(c^T)$.
\end{corollary}
\begin{proof}
It suffices to prove that (i)-(iii) in Lemma \ref{lem:separation_LQ_condition} is equivalent to (a)-(c). Obviously, (i) and (ii) are equivalent to (a) and (b). Finally, we observe that
\begin{equation*} 
\left\{  h(x) = 0  \right\} = \left\{ x_{0} + V u \left|~ u \in \mathbb{R}^{n-1} \right. \right\},
\end{equation*}
where $V \in \mathbb{R}^{n \times (n-1)}$ is a matrix basis of $\mathcal{N}(c^T)$, and hence,
\begin{equation*} 
\begin{split}
F(u) := f|_{\{h=0\}}
&= (x_{0}+Vu)^T A (x_{0}+Vu)+2a^T(x_{0}+Vu)+a_0  \\
&= u^TV^TAVu + 2(x_{0}^TA+a^T)Vu + f(x_0) .
\end{split}
\end{equation*}

When (iii) holds, we have $F(u) > 0$ on $\mathbb{R}^{n-1}$, and hence $V^T A V \preceq 0$ and $w \in \mathcal{R}(V^T A V)$. From
\begin{equation} \label{eq:min_F}
\displaystyle \min_{u \in \mathbb{R}^{n-1}} F(u) = f(x_0) - w^T (V^T A V)^{\dagger} w > 0,
\end{equation}
we know (c) holds true. Conversely, when (c) holds, (\ref{eq:min_F}) holds true, and hence (iii) holds.
\end{proof}

With Lemma \ref{lem:separation_LQ_condition}, Theorem \ref{thm:mutual-sep} establish the mutual separation property for two quadratic hypersurfaces.

\begin{theorem}\label{thm:mutual-sep}
Suppose that both $f$ and $g$ are quadratic functions. If $\{g=0\}$ separates $\{f=0\}$, then $\{f=0\}$	separates $\{g=0\}$ too.
\end{theorem}
\begin{proof}
When $\{g=0\}$ separates $\{f=0\}$, $\{f=0\}$ must be disconnected, and therefore, $f(x)$ or $-f(x)$ has canonical form $-x_1^2+\delta(x_2^2+\cdots +x_m^2)+1, \delta\in\{0, 1\}$. Moreover, by Theorem \ref{thm:separation_reduce_linear}, we know  $-\lambda f(x)+g(x)$ is affine and $\{-\lambda f+g(x)=0\}$ separates $\{f=0\}$. Therefore, Corollary \ref{cor:chckable} implies that $b\in {\rm range}(A)$.

- Case1: ${\rm rank}(A)=1$. Namely, $f(x)=-x_1^2+1$. Since $B = \lambda A \neq 0$ for some $\lambda$ and $b \in {\rm range}(A)$, $g(x)$ has the form $b_{11}x_1^2+b_1x_1+b_0$ with $b_{11} \neq 0$. Observe that $\{f=0\}=\{-1, 1\}$. By the assumption: $\{g=0\}$ separates $\{f=0\}$, we have $g(1)g(-1)<0$. Therefore, equation $b_{11}x_1^2+b_1x_1+b_0=0$ has two distinct real solutions $x_1^-, x_1^+$. One belongs to $(-1, 1)$ and the other lies outside. It means that $f(x_1^-)f(x_1^+)<0$. Namely,  $\{f=0\}$	separates $\{g=0\}$.

- Case2: ${\rm rank}(A)\geq 2$. We are going to prove that: \begin{equation}\label{bhuy56}
\text{There exists } u, v\in \{g=0\} \text{ such that } f(u)f(v)<0.
\end{equation}

Suppose on the contrary that (\ref{bhuy56}) does not hold. This is equivalent to that $\{g=0\}\subset \{f>0\}$ or $\{g=0\}\subset \{f<0\}$. If $\{g=0\}\subset \{f<0\}$, then $\{g=0\}\subset \{-f>0\}$. By S-lemma with equality, $-f(x)+\beta g(x)\geq 0$ for all $x \in \mathbb{R}^n$ for some $\beta \in \mathbb{R}$. It is impossible because $B=\lambda A$ and $A$ takes both negative and positive eigenvalues. With the same argument, we get that $\{g=0\}\subset \{f>0\}$ is impossible. That is to say, (\ref{bhuy56}) holds.

Let $\{g=0\}^1$ be the connected component containing $u$ and $\{g=0\}^2$ be connected component containing $v$. It is obvious that $\{g=0\}^1$ and $\{g=0\}^2$ are not connected; otherwise, there is a curve in $\{g=0\}$ connecting $u$ and $v$, and hence by (\ref{bhuy56}) and Lemma \ref{lem:IVT}, we have $\{g=0\}\cap \{f=0\}\neq \emptyset$, which contradicts to our assumption. Therefore, $\{g=0\}$ has exactly two connected components, $\{g=0\}^1$ and $\{g=0\}^2$. Now, by Lemma \ref{lem:IVT}, the connectedness of $\{g=0\}^1$ and $\{g=0\}^2$, and the fact (\ref{bhuy56}), we get that one connected component of $\{g=0\}$ is a subset of $\{f<0\}$ and that the other one is a subset of $\{f>0\}$. This means that $\{f=0\}$ separates $\{g=0\}$.
\end{proof}

Mutual separation property for quadratic hypersurfaces and Proposition \ref{prop:linear_comb_preserve_separation} together imply the following corollary. It shows that separation between two level sets is preserved under certain linear combination.

\begin{corollary} \label{cor:two_sides_linear_comb}
Suppose that the level set $\{\gamma f  + \delta g = 0\}$ separates the level set $\{\alpha f + \beta g = 0\}$ for some real numbers $\alpha, \beta, \gamma, \delta$ with $\alpha \delta -\beta \gamma \neq 0$. Then
\begin{enumerate}[(a)]
\item if both $\{f=0\}$ and $\{g=0\}$ are quadratic families, mutually separation happens for $\{f=0\}$ and $\{g=0\}$;
\item if one of $f$ and $g$ is affine, then the other must be a quadratic function and the 0-level set of the affine function separates the 0-level set of the quadratic function.
\end{enumerate}
\end{corollary}

\section{Discussion}
 {The advantage of our new tool, called "{\em the separation for the level sets}", was demonstrated by recent results in \cite{Quang-Sheu18,Q-C-S2020,Q-S-X2020,Q-S-S-X2020} with many fundamental mathematical insights. We strongly believe that impotent results, as Brickman's Theorem \cite{Brickman61}, Polyak's Theorems \cite{Polyak98},  will be improved  by applying this  tool.  We hope that our tool can be extended, e.g. $\{g=0\}$ separates $\{f\star 0\}$ on a subset of $\mathbb{R}^n$ (see \cite{Q-S-S-X2020}),  to solving other optimization problems involving  quartic polynomials (see \cite{Q-S-X2020}).}

\bibliographystyle{amsplain}

\end{document}